\documentclass[12pt]{amsart}
\usepackage{amssymb}
\usepackage{amsthm}
\usepackage{amsmath}
\usepackage{enumerate}
\usepackage{graphicx}
\usepackage{mathrsfs}
\usepackage{xypic}

\usepackage[latin1]{inputenc}

\newcommand{\ZZ}{\mathbb{Z}}

\newcommand{\RR}{\mathbb{R}}

\newcommand{\teich}{\mathcal T}

\newcommand{\D}{\mathcal D}

\newcommand{\id}{\mathrm{id}}

\newcommand{\union}{\cup}

\newcommand{\Map}{\mathrm{Map}}
\newcommand{\Out}{\mathrm{Out}}

\theoremstyle{plain}
\newtheorem{theorem}{Theorem}[section]
\newtheorem*{theorem*}{Theorem}
\newtheorem{lemma}[theorem]{Lemma}
\newtheorem{prop}[theorem]{Proposition}
\newtheorem{cor}[theorem]{Corollary}

\theoremstyle{definition}
\newtheorem{defi}[theorem]{Definition}

\theoremstyle{remark}
\newtheorem{remark}[theorem]{Remark}

\begin{document}
\title{The geometry of the handlebody groups I: Distortion}
\author{Ursula Hamenst\"adt and Sebastian Hensel}
\date{April 11, 2011}
\thanks{
  Both authors are partially
  supported by the Hausdorff Center Bonn and the Hausdorff Institut
  Bonn. The second author is supported by the Max-Planck Institut
  f\"ur Mathematik Bonn}
\begin{abstract}
  We show that the mapping class group of a handlebody $V$ of genus at
  least $2$ (with any number of marked points or spots) is 
  exponentially distorted in the mapping class group of its boundary
  surface $\partial V$.
  The same holds true for solid tori $V$ with at least two marked
  points or spots. 
\end{abstract}
\address{\hskip-\parindent
  Ursula Hamenst\"adt and Sebastian Hensel\\ Mathematisches Institut der Universit\"at Bonn \\
  Endenicher Allee 60\\
  D-53115 Bonn, Germany}
\email{ursula@math.uni-bonn.de}
\email{loplop@math.uni-bonn.de}
\maketitle

\section{Introduction}
A handlebody $V_g$ of genus $g$ is a $3$-manifold bounded by a closed
orientable surface $\partial V_g=S_g$ of genus $g$. Explicitly, $V_g$ can be
constructed by attaching $g$ one-handles to a $3$-ball.
Handlebodies are
basic building blocks for closed $3$-manifolds, since any such manifold
can be obtained by gluing two handlebodies along their boundaries.

The handlebody group $\Map(V_g)$ is the subgroup of the mapping 
class group $\Map(\partial V_g)$ of the boundary surface defined by isotopy
classes of those orientation preserving homeomorphisms of $\partial V_g$
which can be extended to homeomorphisms of $V_g$. 
It turns out that $\Map(V_g)$ can be identified with the group of
orientation preserving homeomorphisms of $V_g$ up to isotopy.

The handlebody group is a finitely presented 
subgroup of the mapping class group (compare \cite{Wa98} and \cite{S77}),
and hence it can be equipped with a word norm.
The goal of this article is to initiate an investigation of the coarse geometry
of the handlebody group induced by this word norm. 

The geometry of mapping class groups of surfaces is quite well
understood. Therefore, understanding the geometry of the inclusion homomorphism
$\Map(V_g)\to \Map(\partial V_g)$
may allow to deduce geometric properties of the handlebody group
from geometric properties of the mapping class group.
This task would be particularly easy if the handlebody group was
undistorted in the ambient mapping class group (i.e. if the
inclusion was a quasi-isometric embedding).

Many natural subgroups of the mapping class group are known to be
undistorted. One example is given by groups generated by Dehn twists about
disjoint curves (studied by Farb, Lubotzky and Minsky in \cite{FLM01})
where undistortion can be proved by considering the subsurface
projections onto annuli around the core curves of the Dehn twists.

Another example of undistorted subgroups are mapping class groups of
subsurfaces (compare \cite{MM00} or \cite{H09}).
In this case, the proof of undistortion relies on the 
construction of quasi-geodesics in the mapping class group -- 
either train track splitting
sequences as in \cite{H09} or hierarchy paths
defined by Masur and Minsky in \cite{MM00}. 

Other important subgroups of the mapping class group are known to be
distorted. As one example we mention the Torelli group, which is
exponentially distorted by \cite{BFP07}. A finitely generated subgroup
$H$ of a finitely generated group $G$ is called
exponentially distorted in $G$ if the following holds. On the one
hand, the word norm in $H$ of every element $h\in H$ is coarsely
bounded from above by an exponential of the word norm of $h$ in $G$. On the
other hand, there is a sequence of elements $h_i\in H$ such that the word
norm of $h_i$ in $G$ grows linearly, while the word norm of $h_i$ in
$H$ grows exponentially.

The argument from \cite{BFP07} can be used to show exponential
distortion for other normal
subgroups of the mapping class group as well.
Since the handlebody group is not normal, it cannot be used to analyze
the handlebody group. 

Answering a question raised in \cite{BFP07}, we show that
nevertheless the same conclusion holds true  
for handlebody groups in almost all cases.
\begin{theorem*}
  The handlebody group for genus $g\geq 2$ is exponentially
  distorted in the mapping class group.
\end{theorem*}
Our result is also valid for handlebodies with marked points or
spots; allowing to lower the genus to $1$ if there
are at least two marked points or spots.
In the case of genus $0$ and the solid torus with one marked point 
the handlebody group is obviously undistorted and hence we obtain a
complete classification of distorted handlebody groups.

Apart from the mapping class group, the handlebody group is naturally related to
another important group. 
Namely, the action of $\Map(V_g)$ on the fundamental group 
of the handlebody defines a projection homomorphism 
onto the outer automorphism group $\mathrm{Out}(F_g)$ of a free group
with $g$ generators. 
However, by a theorem of McCullough \cite{Mc85}, the kernel of this
projection homomorphism is infinitely
generated and there are no known tools for transferring properties 
from $\Out(F_g)$ to the handlebody group.

A guiding question for future work is to compare the
geometry of the handlebody group to both mapping class groups and
outer automorphism groups of free groups. 
In particular, in a forthcoming article we shall identify 
quasi-geodesics in the handlebody group and use this
description to shed more light on the geometric nature of the projection to
$\Out(F_g)$ and the inclusion into the mapping class group.

The basic idea for the proof of the main theorem can be sketched in the special case of a
solid torus $V_{1,2}$ with two marked points. 
The handlebody group of a solid torus with one marked
point is infinite cyclic, generated by the Dehn twist $T$ about the unique
essential simple diskbounding curve.
Since point-pushing maps are contained in the handlebody group, the
Birman exact sequence yields that $\Map(V_{1,2})$ is equal to the 
fundamental group of the mapping torus of the once-punctured torus
defined by $T$. The Dehn twist $T$ acts on the fiber $\pi_1(T_{1,1})=F_2$ of the
Birman exact sequence as a
Nielsen twist, therefore in particular as an element of linear growth type.
This implies that the fiber is undistorted in the handlebody group. As
this fiber is exponentially distorted in the mapping class group by
\cite{BFP07}, the handlebody group of a torus with two
marked points is at least exponentially distorted in the corresponding
mapping class group.

In the general case, the argument is more involved since we have no
explicit description of the handlebody group. However, the basic idea
remains to show that parts of the fiber of some suitable Birman exact
sequence are undistorted in the handlebody group.

The upper distortion bound uses a geometric model for the handlebody group.
This model, the \emph{graph of rigid racks}, is similar in spirit to
the train track graph which was used in \cite{H09} to study the
mapping class group. We construct a family of distinguished paths
connecting any pair of points in this graph to each
other. The length of these paths can be bounded using intersection
numbers. The geometric control obtained this way allows to show
the exponential upper bound on distortion.

The paper is organized as follows. In
Section~\ref{sec:undistortion} we recall basic facts about handlebody groups
of genus $0$ and $1$. Section~\ref{sec:marked-points} contains the lower
distortion bound for handlebodies with at least one marked point or
spot. In Section~\ref{sec:closed-surfaces} we show the lower
distortion bound for closed surfaces. 
Section~\ref{sec:disk-systems} introduces a surgery
procedure for disk systems which is important for the construction of paths in the
handlebody group. Section~\ref{sec:racks} is devoted to the construction
of racks, and demonstrates some of their similarities (and differences) to train tracks on surfaces. Section~\ref{sec:rigid-racks} contains the
construction of the geometric model for the handlebody group and a
distinguished family of paths establishing the upper bound on distortion.

\vspace{1 mm}
\textbf{ACKNOWLEDGMENT:} The authors thank Karen Vogtmann for useful
discussions. The authors are also grateful to Lee Mosher for pointing out
the reference \cite{A02}.

\section{Low-complexity cases}\label{sec:undistortion}
As a first step, we analyze the cases of those genus $0$ and $1$ handlebody
groups which turn out to be undistorted.
The results in this section are easy and well-known, and we record
them here for completeness. 

To formulate the results in full generality,
we need to introduce the notion of handlebodies with marked points and spots.
A handlebody of genus $g$ with $k$ marked points and $s$ spots $V_{g,k}^s$ is a
handlebody of genus $g$, together with $s$ pairwise disjoint disks $D_1, \ldots,
D_s$ on its boundary surface $S_g$,  and $k$ pairwise distinct points
$p_1, \ldots, p_k$ in $\partial V_g\setminus(D_1\union\ldots\union D_s)$. 

The mapping class group $\Map(\partial V_{g,k}^s,
p_1,\ldots, p_k, D_1, \ldots, D_s)$ of the boundary surface
(with the same marked points and disks) consists of homeomorphisms of $\partial V_g$ which fix the set $\{p_1, \ldots, p_k\}$ and
restrict to the identity on each of the $D_i$ up to isotopy respecting
the same data. 
Note that this group agrees with the 
mapping class group of the bordered surface obtained by removing the
interior of the marked disks, as these mapping classes have to fix
each boundary component (following the definition in \cite[Section 2.1]{FM11}). 
In the same way as for the case without marked points or spots, the handlebody group 
$\Map(V^s_{g,p}, p_1, \ldots, p_k, D_1, \ldots, D_s)$
is defined as the subgroup of those isotopy classes of
homeomorphisms that extend to the
interior of $V^s_{g,p}$.

All curves and disks are required not to meet any of the marked points. 
A simple closed curve on $\partial V$ is \emph{essential} if it
is neither contractible nor freely homotopic to a marked point.
A disk $D$ in $V$ is called \emph{essential}, if $\partial D
\subset \partial V$ is an essential simple closed curve.

\begin{prop}
  \label{prop:ball-group}
  Let $V=V^s_{0,k}$ be a handlebody of genus $0$, with any number of
    marked points and spots. Then the handlebody group of $V$ is
    equal to the mapping class group of its boundary.
\end{prop}
\begin{proof}
  Let $f:S^2\to S^2$ be any homeomorphism of the standard $2$-sphere
  $S^2\subset\RR^3$ onto itself. We can
  explicitly construct a radial extension $F:D^3\to D^3$ to the standard
  $3$-ball $D^3\subset\RR^3$ by setting
  $F(t\cdot x)=t\cdot f(x)$ for $x \in S^2, t\in [0,1]$.
  Therefore every mapping class group element is contained in the
  handlebody group.
\end{proof}
In particular, the handlebody groups of genus $0$ are undistorted in
the corresponding mapping class groups.
Similarly, for a solid torus with at most one marked point or spot, the
handlebody group can be explicitly identified and turns out to be
undistorted. 

To this end, suppose $V$ is a solid torus with at most one marked point
($V=V_{1,0}$ or $V=V_{1,1}$) or with one marked spot ($V=V_1^1$). 
Let $\delta$ be an essential simple closed
curve on the boundary torus of $V$ that bounds a disk in $V$.
The curve $\delta$ is uniquely determined up to isotopy.
\begin{prop}
  The handlebody group of $V$ is the stabilizer of $\delta$ in
  the mapping class group. In particular, it is undistorted in the
  mapping class group.

  Thus, if $V = V_{1,0}$ or $V=V_{1,1}$, then the
  handlebody group is cyclic and generated by the Dehn twist
  about $\delta$.

  If $V=V_1^1$, then the handlebody group
  is the free abelian group of rank $2$ which is generated by the Dehn
  twist about $\delta$ and the Dehn twist about the spot. 
\end{prop}
\begin{proof}
  The handlebody group fixes the set of isotopy classes of
  essential disks in $V$. 
  Since $\delta$ is the unique diskbounding curve up to isotopy,
  $\Map(V)$ therefore is contained in the stabilizer of $\delta$. On
  the other hand, the disk bounded by $\delta$ cuts $V$ into a spotted
  ball. Hence, by Proposition~\ref{prop:ball-group} the handlebody
  group $\mathrm{Map}(V)$ contains the stabilizer of $\delta$.

  If $V=V_{1,0}$ or $V_{1,1}$, the complement of $\delta$ in $\partial
  V$ is an annulus (possibly with a puncture). From this, it is
  immediate that the handlebody group is generated by the Dehn twist
  about $\delta$.

  If $V=V_1^1$, the same argument shows that then the
  handlebody group is 
  generated by the Dehn twist about $\delta$ and the spot. It is clear
  that these mapping classes commute.

  Since stabilizers of simple closed curves are known to be undistorted
  subgroups of the mapping class group (compare \cite{MM00} or
  \cite{H09b}), the handlebody group of a solid
  torus with at most one spot or marked point is undistorted. 
\end{proof}

\section{Handlebodies with marked points}\label{sec:marked-points}
In this section we describe the lower bound for distortion of handlebody groups
with marked points. We begin with the case of genus $g\geq 2$ with a
single marked point. The case of several marked points or spots will
be an easy consequence of this result. The case of a torus with
several marked points requires a different argument which will be
given at the end of this section.
\begin{theorem}\label{distortion-point}
  Let $V=V_{g,1}$ be a handlebody of genus $g\geq 2$ with one marked
  point, and let $\partial V=S_{g,1}$ be its boundary surface. 
  Then the handlebody group $\Map(V) < \Map(\partial V)$ is
  at least exponentially distorted.
\end{theorem}
The proof is based on the relation between  
the mapping class group of a closed surface $S_g$ and the mapping
class group of a once-punctured surface $S_{g,1}$.
We denote the marked point of $\partial V=S_{g,1}$ by $p$, and 
we will often denote the mapping class group of $S_{g,1}$ by $\Map(S_g,p)$.

Recall the definition of the \emph{point-pushing map} ${\mathcal P}:\pi_1(S,
p)\to\Map(S,p)$. Namely, let $\gamma:[0,1]\to S$ be a loop in $S$
based at $p$. Then there is an isotopy $f_t:S\to S$ supported in a small
neighborhood of the loop $\gamma[0,1]$ such that
$f_0 = \id$, and $f_t(p)=\gamma(t)$. To see this, note that
locally around $\gamma(t_0)$ such an isotopy certainly exists (for
example, since any orientation preserving homeomorphism of the disk is
isotopic to the identity). The
image of $\gamma$ is compact, and hence the desired isotopy can be
pieced together from finitely many such local isotopies.
The endpoint $f_1$ of such an isotopy is a homeomorphism of
$(S,p)$. We call its isotopy class the point pushing map
$\mathcal{P}(\gamma)$ along $\gamma$. It depends only on the homotopy
class of $\gamma$.

The image of the point pushing map is contained
in the handlebody group $\Map(V, p)$ -- to see this, simply define the local
version by pushing a small half-ball instead of a disk.

By construction, the image of the point pushing map lies in the kernel
of the forgetful homomorphism $\Map(S, p) \to \Map(S)$ induced by the
puncture forgetting map $(S,p)\to (S,S)$. In fact this is all of
the kernel, compare \cite{Bi74}.
\begin{theorem}[Birman exact sequence]\label{birman}
  Let $S$ be a closed oriented surface of genus $g\geq 2$ and $p \in S$ any
  point. The sequence
  $$\xymatrix{ 1 \ar[r] & \pi_1(S,p) \ar^{\mathcal P}[r] & \Map(S,p)
   \ar[r] & \Map(S) \ar[r] & 1}$$
  is exact. 
\end{theorem}
The point pushing map is natural in the sense that
\begin{equation}\label{naturality}
\mathcal{P}(f\alpha) = f \circ \mathcal{P}(\alpha) \circ f^{-1}
\end{equation}
for each $f \in \Map(S, p)$ (see \cite{Bi74} for a proof of this
fact).

The Birman exact sequence corresponds to the
relation between the inner and the outer automorphism group of $\pi_1(S,p)$:
  $$\xymatrix{ 1 \ar[r] & \pi_1(S,p) \ar^{\mathcal P}[r]\ar^\cong[d] & \Map(S,p)\ar^\cong[d]
   \ar[r] & \Map(S)\ar^\cong[d] \ar[r] & 1 \\
   1 \ar[r] & \mathrm{Inn}(\pi_1(S,p)) \ar[r] & \mathrm{Aut}(\pi_1(S, p))
   \ar[r] & \mathrm{Out}(\pi_1(S, p)) \ar[r] & 1}$$
where $\pi_1(S, p)$ can be identified with its inner automorphism
group because it has trivial center, and the other two isomorphisms
are given by the Dehn-Nielsen-Baer theorem.
In other words, we have the following.
\begin{lemma}\label{point-pushing-action}
  Let $[\gamma],[\alpha]\in\pi_1(S, p)$ be two loops at $p$. Then
  $$\mathcal{P}(\alpha)(\gamma) = [\alpha] * [\gamma] *
  [\alpha]^{-1}$$
  where $*$ denotes concatenation of loops, and takes place left-to-right.
\end{lemma}
Now we are ready to give the proof of the main theorem of this section.
\begin{proof}[Proof of Theorem~\ref{distortion-point}]
  Let $\delta$ be a separating simple closed curve on $S$ such that
  one component of $S \setminus \delta$ is a bordered torus $T$ with
  one boundary circle,
  and such that $\delta$ bounds a disk $\D$ in the handlebody $V$.
  Without loss of generality we assume that the base point $p$
  lies on $\delta$. 
  \begin{figure}[h!]
    \centering
     \includegraphics[width=0.7\textwidth]{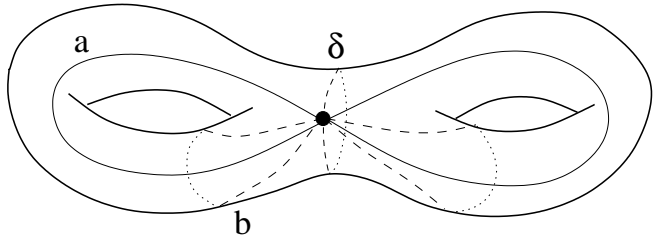}
    \caption{The setup in the proof of
      Theorem~\ref{distortion-point}. Generators for the fundamental
      group of the handlebody are drawn solid, the loops
      extending these to a generating set of $\pi_1(S,p)$ are drawn dashed.}
    \label{fig:point}
  \end{figure}

  Choose loops $a,b$ based at $p$ which generate the fundamental group
  of $T$ and such that $b$ bounds a disk in $V$ (and hence $a$ does not).
  Extend $a, b$ to a generating set of the fundamental group of
  $\pi_1(S,p)$ by adding loops in the complement of $T$ (see
  Figure~\ref{fig:point}).
  Let $f \in \Map(S, p)$ be a mapping class such that $f(a) = a^2 * b$
  and $f(b) = a*b$ which preserves $\delta$ and acts as the identity on
  $S\setminus T$. Such an $f$ can for example be obtained as the
  composition of suitably
  oriented Dehn twists along $a$ and $b$. 

  Define $\Phi_k = \mathcal{P}(f^k a)$. By Equation~(\ref{naturality}), in the mapping
  class group $\Map(S, p)$ we have $\Phi_k=f^k\mathcal{P}(a)f^{-k}$,
  and hence the word norm of $\Phi_k$ in the mapping class group with
  respect to any generating set grows
  linearly in $k$.

  On the other hand, consider the map
  $$\Map(V, p) \stackrel{\pi}{\to} \mathrm{Aut}(\pi_1(V, p)) = \mathrm{Aut}(F_g)$$
  defined by the action on the fundamental group. 
  Lemma~\ref{point-pushing-action} implies that
  $\Phi_k$ acts on $\pi_1(S,p)$ as conjugation by $f^k(a)$.
  To compute the action of $\pi(\Phi_k)$ on $\pi_1(V,p)$, denote the
  projection of the fundamental group of the surface $S$ to the
  fundamental group of the handlebody by
  $P:\pi_1(S,p) \to \pi_1(V,p)$.

  Since $b$ bounds a disk in $V$, its projection vanishes: $P(b)=0$.
  The generator $a$ of $\pi_1(S,p)$ projects to a primitive element in
  $\pi_1(V,p)$, $P(a) = A$.
  Hence $P(f^k(a)) = A^{N_k}$ for some $N_k > 0$.
  The choice of $f$ guarantees that we have $N_k \geq 2^k$.
  Since the point pushing map is natural with respect to the
  projection to the handlebody, $\pi(\Phi_k)$ acts on $\pi_1(V,p)$ as
  conjugation by $A^{N_k}$. 

  In other words, as an element of $\mathrm{Aut}(F_g)$ the projection $\pi(\Phi_k)$ is the
  $N_k$--fold power of the conjugation by $A$. Since conjugation
  by $A$ is an infinite order element in $\mathrm{Aut}(F_g)$ and all
  infinite order elements have positive translation length (compare
  \cite[Theorem 1.1]{A02}) this implies that the word norm of
  $\pi(\Phi_k)$ grows exponentially in $k$. 
  As $\pi: \Map(V, p) \to\mathrm{Aut}(F_g)$ is a surjective homomorphism between
  finitely generated groups, it is Lipschitz with respect to any
  choice of word metrics. Therefore, the word norm of $\Phi_k$ in
  $\Map(V, p)$ also grows exponentially in $k$. This shows the theorem. 
\end{proof}
\begin{remark}
  The proof we gave extends verbatim to the case of the pure
  handlebody group of a handlebody of genus $g\geq 2$ with several marked points and any
  number of spots (just move everything but one marked point into the
  complement of $T$). Here, the pure handlebody group is the subgroup
  of those mapping classes which send each marked point to itself.
  Since this group has finite index in the full handlebody group, the
  proof also shows that handlebody
  groups with several marked points and any number of spots are at
  least exponentially distorted if the genus is at least $2$. 
\end{remark}
As a next case, we consider handlebody groups of handlebodies with
spots instead of marked points.
\begin{cor}\label{distortion-spot}
  Let $V=V_g$ be a genus $g\geq 2$ handlebody and let $D
  \subset \partial V$ be a spot. 
  Then the handlebody group $\Map(V, D) < \Map(\partial V, D)$ of the
  spotted handlebody is at least exponentially distorted.
\end{cor}
\begin{proof}
  Note that there is a commutative diagram with surjective projection homomorphisms
  $$\xymatrix{0\ar[r]& \ar[r]\ar[d]^{=}\left<T\right>&\ar[d]\ar[r] \Map(V, D) & \ar[d]\ar[r]\Map(V, p)&0\\
  0\ar[r]&\ar[r]\left<T\right>&\ar[r] \Map(\partial V, D) &
  \ar[r]\Map(\partial V, p)&0}$$
  induced by collapsing the marked spot to a point. The kernel of such a
  projection homomorphism is infinite
  cyclic and generated by the Dehn twist $T$ about the spot. In
  particular, every element $g$ in $\Map(\partial V,p)$ lifts to an element in
  $\Map(\partial V,D)$, and if $g \in \Map(V,p)$ then the lift is contained in the
  handlebody group $\Map(V,D)$.
  These lifts are well-defined up to the Dehn twist $T$ which lies in
  the handlebody group and acts trivially on $\pi_1(V,p)$.

  Choose any lift $\widetilde{f}$ of the element $f$ used in the proof of
  Theorem~\ref{distortion-point}.
  Let $\widetilde{\Phi}$ be a lift of the point pushing map
  $\Phi_0$ defined in the proof of Theorem~\ref{distortion-point}, and define $\widetilde{\Phi}_k =
  \widetilde{f}^k\widetilde{\Phi}\widetilde{f}^{-k}$. Note that these
  elements are lifts of the elements $\Phi_k$ and therefore contained in the
  handlebody group.
 
  Now $\widetilde{\Phi}_k$ has word norm in $\Map(S, D)$ again 
  bounded linearly in $k$. As elements of the spotted handlebody
  group the word norm of $\widetilde{\Phi}_k$ grows exponentially in
  $k$, as this is true for the $\Phi_k$.
\end{proof}
\begin{remark}
  Again, the same proof works for handlebodies with more than one
  spot and any number of marked points.
\end{remark}
As a last case, we consider the handlebody of a torus with more than
one marked point.
\begin{theorem}
  Let $V=V_{1,n}$ be a solid torus with $n\geq2$ marked points. Then the
  handlebody group $\Map(V)$ is at least exponentially distorted in
  $\Map(\partial V)$.
\end{theorem}
\begin{proof}
  The strategy of this proof is similar to the preceding ones. We
  consider the Birman exact sequence for pure mapping class groups and pure
  handlebody groups. 
  $$\xymatrix{1\ar[r]& \ar[r]^-{\mathcal{P}}
    \pi_1\mathcal{C}_n&P\Map(\partial V,p_0,p_1,\ldots,p_n)\ar[r]&
  \Map(\partial V,p_0)\ar[r]&1\\
  1\ar[r]&\pi_1\mathcal{C}_n\ar[r]\ar[u]_{=}&P\Map(V,p_0,p_1,\ldots,p_n)\ar[r]\ar[u]&
  \ZZ=\left<T\right>\ar[r]\ar[u]&1}$$
  where $\mathcal{C}_n$ denotes the configuration space of $n$ points
  in $\partial V\setminus\{p_0\}$, and $T$ the Dehn twist along the (unique)
  disk $\delta$ on $\partial V\setminus\{p_0\}$. An element of
  $\pi_1\mathcal{C}_n$ can be viewed as an $n$-tuple of parametrized loops
  $\gamma_i$, where $\gamma_i$ is based at $p_i$ (subject to the
  condition that at each point in time, the values of all these loops
  are distinct). Note that the pure mapping class group
  $P\Map(\partial V,p_0,p_1,\ldots,p_n)$ acts on $\mathcal{C}_n$ by acting on
  all component loops.
  The map $\mathcal{P}$ is the generalized point pushing map,
  pushing all marked points simultaneously along the loops
  $\gamma_i$.
  The map $\mathcal{P}$ is natural with respect to the action of
  $P\Map(\partial V,p_0,p_1,\ldots,p_n)$ in the sense that
  $\mathcal{P}(f\gamma)=f\circ\mathcal{P}(\gamma)\circ f^{-1}$.

  Every element of $P\Map(V,p_0,p_1,\ldots,p_n)$ can be 
  written in the form $\mathcal{P}(\gamma)\cdot\widetilde{T}^l$, where
  $\gamma$ denotes an $n$-tuple of loops, and $\widetilde{T}$ is some
  (fixed) lift of the Dehn twist $T$. In this description, the
  multiplicity $l$ and the
  homotopy class of the $n$-tuple of loops $\gamma$ is
  well-defined. Now note that
  \begin{equation}\label{composition}
    \left(\mathcal{P}(\gamma)\cdot\widetilde{T}^l\right)\cdot\left(\mathcal{P}(\gamma')\cdot\widetilde{T}^{l'}\right)
    =
    \mathcal{P}(\gamma)\cdot\mathcal{P}\left(\widetilde{T}^{l'}(\gamma')\right)\widetilde{T}^{l+l'}
  \end{equation}
  \begin{equation*}
    = \mathcal{P}\left(\widetilde{T}^{l'}(\gamma')*\gamma\right)\widetilde{T}^{l+l'}
  \end{equation*}
  by the naturality of $\mathcal{P}$ and the fact that $\mathcal{P}$
  is a homomorphism (note that concatenation of loops is executed
  left-to-right, while composition of maps is done right-to-left).

  Choose an element $\beta \in \pi_1(\partial V, p_0)$ which extends $\delta$
  to a basis of $\pi_1(\partial V, p_0)=F_2$. 
  Note that then $\beta$ is a generator of the fundamental group
  $\pi_1(V, p_0)=\ZZ$ of the solid torus $V_1$.
  We also choose loops $\beta_i \in \pi_1(\partial V,p_i)$ for all
  $i=1,\ldots, n$ which are freely homotopic to $\beta$.
  These loops give an identification of $\pi_1(V,p_i)$ with $\ZZ$.

  Define a map $b:P\Map(V,p_0,p_1,\ldots,p_n) \to \ZZ$ as follows. Let $\varphi = \mathcal{P}(\gamma)\cdot\widetilde{T}^l$ be any
  element of the pure handlebody group. Each component loop $\gamma_i$ of
  $\gamma$ defines a loop in $\pi_1(V,p_i)$
  (which might be trivial).
  This loop is homotopic to the $k_i$-th power of $\beta_i$ for some
  number $k_i$.
  Associate to $\varphi$ the sum of the $k_i$. 

  Now choose any generating set $\gamma^1, \ldots, \gamma^N$ of
  $\pi_1\mathcal{C}_n$. Then the pure handlebody group $P\Map(V,p_0,p_1,\ldots,p_n)$ is
  generated by $\mathcal{P}(\gamma^j)$ and $\widetilde{T}$.
  We claim that there is a constant $k_0$, such that 
  \begin{equation}
    \label{eq:decrease}
     b\left(\varphi\cdot\mathcal{P}(\gamma^i)\right) \geq b(\varphi)-k_0
  \end{equation}
  Namely, by equation~(\ref{composition}), we have to compare the
  projections of the components of
  $$\gamma\quad\quad\mbox{and}\quad\quad\widetilde{T}^{l}(\gamma^j)*\gamma$$
  to each of the $\pi_1(V,p_i)$.
  However, applying $\widetilde{T}$ 
  does not change this projection.
  Since $\gamma^j$ is one of finitely many
  generators, there is a maximal number of occurrences of the
  projection of $\beta_i$
  which can be canceled by adding the projection of $\gamma^j$. This shows inequality~(\ref{eq:decrease}).

  Now we can finish the proof using a similar argument as in the proof of
  Theorem~\ref{distortion-point}. Namely, choose again $f$ a
  pseudo-Anosov element with the property that applying $f$ 
  multiplies the number of occurrences of $\beta_i$ by $2$ in all
  $\pi_1(\partial V,p_i)$. Then 
  $\mathcal{P}(f^k\beta)$ has length growing linearly in the mapping
  class group, while $b(f^k\beta)$ grows exponentially. By
  inequality~(\ref{eq:decrease}) this implies that the word norm in the pure
  handlebody group also grows exponentially. Since the pure handlebody
  group has finite index in the full handlebody group the theorem follows.
\end{proof}
\begin{remark}
  The same argument that extends Theorem~\ref{distortion-point} to
  Corollary~\ref{distortion-spot} applies in this case and shows that
  also all torus handlebody groups with at least two spots or marked
  points are exponentially distorted.
\end{remark}

\section{Handlebodies without marked points}\label{sec:closed-surfaces}
In this section we complete the proof of the exponential lower bound
on the distortion of the handlebody groups by showing
that the handlebody group of a
handlebody of genus $g \geq 2$ without marked points or spots is 
distorted in the mapping class group.

For genus $g\geq 3$, the idea is to replace the point pushing used
in the proofs above by pushing a subsurface around the handlebody. The
resulting handlebody group element does not induce a conjugation on
$\pi_1(V, p)$, but instead induces a partial conjugation on the
fundamental group of the complement of the pushed subsurface. Since $g\geq 3$, 
such an element projects to a nontrivial element in the outer automorphism group of $F_g$.
Then a similar reasoning as in Section~\ref{sec:marked-points} applies. The case of genus $2$
requires a different argument and will be given at the end of this section.
\begin{theorem}\label{distortion-closed}
  For a handlebody $V=V_g$ of genus $g \geq 3$, the handlebody group
  $\Map(V)$ is at least 
  exponentially distorted in the mapping class group $\Map(\partial V)$.
\end{theorem}
\begin{proof}
  Choose a curve $\delta$ which bounds a disk $\D$, such that
  $V\setminus\D$ is the union of a once-spotted genus 2 handlebody
  $V_1$ and a once-spotted genus $g-2$ handlebody $V_2$.
  Denote the boundary of $V_i$ by $S_i$, and choose a basepoint $p \in
  \delta$. This defines a free decomposition of the fundamental group
  of the handlebody 
  $$F_g = \pi_1(V,p) = \pi_1(V_1,p) * \pi_1(V_2,p) = F_2 * F_{g-2}.$$
  We denote by $\Map(S_i, \delta)$ the mapping class group of
  the bordered surface $S_i$, emphasizing that each such mapping class
  has to fix $\delta$ pointwise. 
  The stabilizer of $\delta$ in the mapping class group of
  $S$ is of the form
  $$G_S = \left.\Map(S_1, \delta) \times \Map(S_2, \delta)\right/\thicksim$$
  where the equivalence relation $\thicksim$ identifies the Dehn twist
  about $\delta$ in $\Map(S_1, \delta)$ and $\Map(S_2,
  \delta)$. 
  Note that the Dehn twist about $\delta$ lies in the handlebody group and
  acts trivially on $\pi_1(V,p)$.
  Therefore, the stabilizer of $\delta$ in the handlebody
  group is of the form
  $$G_V = \left.\Map(V_1, \D) \times \Map(V_2, \D)\right/\thicksim.$$
  In particular, the handlebody group $\Map(V_1, \D)$
  injects into $G_V$. 
  There is a homomorphism $G_V \to \mathrm{Aut}(F_2) \times
    \mathrm{Aut}(F_{g-2})$ induced by the actions of $\Map(V_i,p)$ on
    $\pi_1(V_i,p)$.  
  This homomorphism is natural with respect to the inclusion
  $\mathrm{Aut}(F_2) \times \mathrm{Aut}(F_{g-2}) \to
  \mathrm{Aut}(F_g)$ defined by the free decomposition of
  $\pi_1(V,p)$ given above. It is also natural with respect to the inclusion 
  $\mathrm{Aut}(F_2)\to\mathrm{Aut}(F_2) \times \mathrm{Aut}(F_{g-2})$
  defined by
  $\Map(V_1,\D)\to G_V$.
  Summarizing, we have the following commutative diagram.
  $$\xymatrix{
    \ar[r] \Map(S_1, \delta) & \ar[r] G_S & \ar[r] \Map(S, p) & \Map(S) \\
    \ar[r]\ar[u]\ar[d] \Map(V_1, \D) & \ar[r]\ar[u]\ar[d] G_V & \ar[r]\ar[u]\ar[d] \Map(V, p) & \ar[u]\ar[d]\Map(V)
    \\
    \mathrm{Aut}(F_2)\ar[r] & \ar[r] \mathrm{Aut}(F_2) \times \mathrm{Aut}(F_{g-2}) & \ar[r]\mathrm{Aut}(F_g) & \mathrm{Out}(F_g)
    }$$
    Let $\widetilde{\Phi}_k \in \Map(V_1, \D)$ be the elements
    constructed in the proof of Corollary~\ref{distortion-spot}.
    The image of $\widetilde{\Phi}_k$ in
    $\mathrm{Aut}(F_2) \times \mathrm{Aut}(F_{g-2})$ is the
    $N_k$--th power of a conjugation in the
    free factor $F_2$ defined by $V_1$, and the identity on the
    free factor $F_{g-2}$ defined by $V_2$, where $N_k \geq 2^k$.
    In other words, this projection is a $N_k$--th iterate of a
    partial conjugation. Therefore, it projects to a
    nontrivial element of infinite order in $\mathrm{Out}(F_g)$. From there,
    one can finish the proof using the argument in the proof of Theorem~\ref{distortion-point}.
\end{proof}

The last case is that of a genus $2$ handlebody without marked points
or spots. In this case, the strategy is to use the distortion
of the handlebody group of a solid torus with two spots to produce
distorted elements in the stabilizer of a nonseparating disk in the genus $2$
handlebody. 

To make this precise, we use the following construction. Let $V$ be a
genus $2$ handlebody and $S$ its boundary surface. Choose a
nonseparating essential simple closed curve $\delta$ that bounds a
disk $\D$ in $V$. 
Cutting $S$ at $\delta$ yields a torus $S^2_1$ with two
boundary components $\delta_1$ and $\delta_2$. 
Choose once and for all a continuous map $S^2_1 \to S$
which maps both $\delta_1$ and $\delta_2$ to $\delta$ and which
restricts to a homeomorphism
$$S^2_1\setminus(\delta_1\union\delta_2) \to S\setminus\delta.$$
The isotopy class of such a map depends on choices, but we fix
one such map for the rest of this section. This map induces induces a homomorphism
$$\Map(S^2_1) \to \mathrm{Stab}_{\Map(S)}(\delta)$$
since the homeomorphisms and isotopies used to define the mapping
class group $\Map(S^2_1)$  of the torus $S_1^2$ have to fix $\delta_1$
and $\delta_2$ pointwise and therefore extend to $S$.

Since $\delta$ bounds a disk, an analogous construction
works for the handlebody groups, and we obtain
$$\Map(V^2_1) \to \mathrm{Stab}_{\Map(V)}(\D).$$
Let $p\in\delta$ be a base point, and let $a,b$ be smooth embedded loops in $S$  
with the following properties (compare
Figure~\ref{loops}). 
\begin{figure}[h!]
  \centering
  \includegraphics[width=0.7\textwidth]{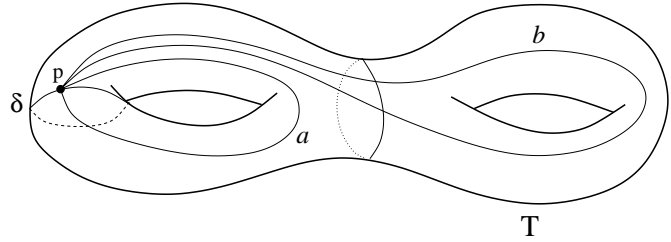}
  \caption{The setting for a genus $2$ handlebody. }
  \label{loops}
\end{figure}
\begin{enumerate}[i)]
\item The projections $A$ and $B$ of $a$ and $b$ to $\pi_1(V,p)$ form
  a free basis of $\pi_1(V,p)=F_2$.
\item The loops $a$ and $b$ intersect
$\delta$ exactly in the basepoint $p$.
\item The loop $a$ hits $\delta$ from different sides
at its endpoints, while $b$ returns to the same side.
\end{enumerate}
On the surface $S_1^2$ obtained by cutting $S$ at $\delta$, the loop $a$
defines an arc from one boundary component to the other, while
$b$ defines a loop. By slight abuse of notation we
will denote these objects by the same symbols. We choose the initial
point of the loop $b$ as base point of this cut-open surface, and call
it again $p$. Then the projection $B$ of $b$ to the spotted solid
torus $V_1^2$ is a generator of its fundamental group $\pi_1(V_1^2,p)=\ZZ$.

Now consider the torus $T'\subset S$ with one boundary component 
obtained as the tubular neighborhood of $a\union\delta$ in $S$ (compare Figure~\ref{loops}
for the situation). The complement of $T'$ in $S$ again is a torus with one boundary
component which we denote by $T$. Choose a reducible homeomorphism $f$
of $V_1^2$ which preserves $T$ and restricts to a pseudo-Anosov
homeomorphism $f$ on the torus $T \subset S$ 
with the property that the projection of the loop
$f^k(b)$ to $\pi_1(V^2_1)$ is $B^{N_k}$, for $N_k \geq 2^k$.
Such an element can be constructed explicitly as in the
proof of Theorem~\ref{distortion-spot}.
In particular, we may assume that $f$ fixes the arc $a$ pointwise.

Consider now as in the proof of Theorem~\ref{distortion-spot} the map
that collapses the boundary components of $V_1^2$ to marked points.
On this solid torus $V_{1,2}$ with two marked points, $a$ defines an
arc from marked point to marked 
point, and $b$ defines a based loop at one of the marked points which
we again use as base point for this surface. 
Let $P=\mathcal{P}(b)$ be the point pushing map on $V_{1,2}$ defined by
$b$, and let $\widetilde{P}$ be any lift of this point-pushing map to
the surface $S_1^2$ with boundary. As before, $\widetilde{P}$ is an
element of the handlebody group. We define
$$\Phi_k = f^k\circ\widetilde{P}\circ f^{-k}.$$
\begin{lemma}
  $\Phi_k$ is an element of the handlebody group of $V_1^2$. 
  $\Phi_k(B)$ is homotopic to $B$ as a loop based at $p$ in the
  handlebody $V_1^2$, and $\Phi_k(A)$ is homotopic, as an arc relative
  to its endpoints, to $A*B^{N_k}$ in $V_1^2$.
\end{lemma}
\begin{proof}
  $\Phi_k$ projects to the point-pushing map along $f^k(b)$ on the
  solid torus with two marked points $V_{1,2}$ obtained by collapsing the
  boundary components of $V_1^2$. Hence, $\Phi_k$ is the lift of a
  handlebody group element and therefore lies in the handlebody group itself
  (see the discussion in the proof of
  Theorem~\ref{distortion-spot}). 
  This yields the first claim. 

  To see the other claims, we can work in the solid torus
  $V_{1,2}$ with two marked points, as the projection from
  $V_1^2$ to $V_{1,2}$ that collapses the spots to marked points
  induces a isomorphism on fundamental groups.

  Here by construction $\Phi_k$ projects to the point-pushing map
  along $f^k(b)$. Lemma~\ref{point-pushing-action} now
  implies that this projection acts as conjugation by $B^{N_k}$ on
  the fundamental group,
  giving the second claim.

  By construction of $f$, the arc $a$ and the loop $b_k=f^k(b)$ only
  intersect at the base point.
  The loop $b_k$ is a simple curve and thus there is an embedded tubular
  neighborhood of $b_k$ on $V_{1,2}$ which is orientation preserving
  homeomorphic to $[0,1]/(0\thicksim 1)\times[-1,1]=S^1\times[-1,1]$ and such that
  $S^1\times\{0\}$ is the loop $b_k$. After perhaps reversing the
  orientation of $b_k$ and performing an isotopy, we may assume that the
  intersection of $a$ with this tubular neighborhood equals
  $\{0\}\times[-1,0]$.

  Since $b_k$ is simple, the point pushing map along $b_k$ is isotopic
  to the map supported on the tubular neighborhood which is defined by
  \begin{eqnarray*}
    (x,t) \mapsto (x+(t+1), t)  &\mbox{for}& t\in[-1,0]\\
    (x,t) \mapsto (x-t, t)&\mbox{for}& t\in[0,1]
  \end{eqnarray*}

  This implies that the point pushing map acts on the homotopy class
  of $a$ by concatenating $a$ with the loop $f^k(b)$ (up to possibly
  changing the orientation of $a$). Since $f^k(b)$ projects
  to $B^{N_k}$ in the handlebody, this implies the last claim of the lemma.
\end{proof}
\begin{theorem}
  The handlebody group of a genus $2$ handlebody is at least
  exponentially distorted.
\end{theorem}
\begin{proof}
  We use the notation from the construction described above. Consider
  the image $\Psi_k$ of $\Phi_k$ in the stabilizer of $\D$ in
  the handlebody group $\Map(V_2)$. By construction, $\Psi_k$ fixes the
  curve $\delta$ pointwise and therefore acts on
  $\pi_1(V,p)$. By the preceding lemma, this action is given by
  \begin{eqnarray*}
    A &\mapsto& A*B^{N_k}\\
    B &\mapsto& B
  \end{eqnarray*}
  Therefore, $\Psi_k$ acts as the $N_k$-th power of a simple Nielsen
  twist on $F_2$. In particular, it projects to the $N_k$-th power of
  a nontrivial element in $\Out(F_2)$. From here, one can finish the
  proof as for the preceding distortion theorems.
\end{proof}

\section{Disk exchanges and surgery paths}
\label{sec:disk-systems}
In this section we study disk systems in handlebodies and
introduce certain types of surgery operations for disk systems. 
These surgery operations form the basis for the construction of
distinguished paths in the handlebody group (see
Lemma~\ref{lemma:inductive-distance-control}).

In the sequel we always consider a handlebody $V$
of genus $g\geq 2$ with a finite
number $m$ of marked points on its boundary $\partial V$.
The discussion remains valid if some of the marked points are replaced
by spots.

\begin{defi}
A \emph{disk system} for $V$ is a set of essential disks in $V$ which
are pairwise disjoint and non-homotopic.
A disk system is called
\emph{simple} if all of its complementary components are simply
connected. It is called \emph{reduced} if it is simple and has a
single complementary component.   
\end{defi}
We usually consider disk systems only up to isotopy.
For a handlebody of genus $g$, a
reduced disk system consists of precisely $g$ non-separating
disks. The complement of a reduced disk system in $V$ is a ball with
$2g$ spots (and possibly some marked points). 
The boundary of a reduced disk system
is a multicurve in $\partial V$ with $g$ components which 
cuts $\partial V$ into a $2g$-holed sphere (with some number of marked
points). The handlebody group
acts transitively on the set of isotopy classes of reduced disk systems.

We say that two disk systems $\D_1,\D_2$ are in 
\emph{minimal position} if their boundary multicurves 
intersect in the minimal number of points and if
every component of $\D_1\cap
\D_2$ is an embedded arc
in $\D_1\cap \D_2$ 
with endpoints in $\partial \D_1
\cap \partial \D_2$.
Disk systems can always be put in minimal position by applying
suitable isotopies. 
In the sequel we always assume that disk systems are in minimal position.

Note that the minimal position of disks behaves differently than the
normal position of sphere systems as defined in
\cite{Ha95}. Explicitly, let $\Sigma$ be a reduced disk system and $D$
an arbitrary disk. Suppose $D$ is in minimal position with respect to $\Sigma$.
Then a component of $D\setminus\Sigma$ may have several boundary
components on the same side of a disk in $\Sigma$. In addition, the
collection of components of $D\setminus\Sigma$ does not determine the
disk $D$ uniquely. 

Let $\D$ be a disk system. An \emph{arc relative to
  $\D$} is a continuous embedding $\rho:[0,1]\to\partial V$ such
that its endpoints $\rho(0)$ and $\rho(1)$ are contained in $\partial\D$.
An arc $\rho$ is called \emph{essential} if
it cannot be homotoped into $\partial \D$ with fixed endpoints and if
the number of intersections of $\rho$ with $\partial \D$ is minimal in
its isotopy class.

Choose an orientation of the curves in $\partial \D$. 
Since $\partial V$ is oriented, this choice determines a left and a
right side of a component $\alpha$ of $\partial \D$ in a small annular
neighborhood of 
$\alpha$ in $\partial V$. We then say that an endpoint $\rho(0)$ (or
$\rho(1)$) of an arc $\rho$ \emph{lies to the right (or to the left)
  of $\alpha$}, if a small neighborhood 
$\rho([0,\epsilon])$ (or $\rho([1-\epsilon,1])$) of this
endpoint is contained in the right (or left) side of $\alpha$ in a
small annulus around $\alpha$.
A \emph{returning arc relative to $\D$} is an arc both of whose
endpoints lie on the same side of some boundary $\partial D$ of a disk $D$
in $\D$, and whose interior is disjoint from $\partial \D$.

Let $E$ be a disk which is not disjoint from $\D$. 
An \emph{outermost arc} of $\partial E$ relative to $\D$ is
a returning arc $\rho$ relative to $\D$ such that
there is a component $E^\prime$ of $E\setminus\D$ whose boundary is
composed of $\rho$ 
and an arc $\beta\subset D$. The interior of $\beta$
is contained in the interior of $D$. We call such a disk $E^\prime$
an \emph{outermost component} of $E\setminus\D$.

For every disk $E$ which is not disjoint
from $\D$ there are 
at least two distinct outermost components $E^\prime,E^{\prime\prime}$
of $E\setminus\D$. Every outermost arc of a disk is a returning
arc. However, there may also be components of $\partial E\setminus\D$ which
are returning arcs, but not outermost arcs. 
For example, a component of $E\setminus \D$ may be a rectangle bounded by two 
arcs contained in $\D$ and two subarcs of $\partial E$ with endpoints
on $\partial \D$ which are homotopic to a returning arc relative to
$\partial \D$.   

Let now $\D$ be a simple disk system and let $\rho$ be a
returning arc whose endpoints are contained in the boundary of some
disk $D \in \D$. 
Then $\partial D \setminus \{\rho(0),\rho(1)\}$ is the union of
two (open) intervals $\gamma_1$ and $\gamma_2$. Put $\alpha_i =
\gamma_i\union\rho$. Up to isotopy, $\alpha_1$ and $\alpha_2$ are
simple closed curves which are disjoint from $\D$ (compare
\cite{St02} and \cite{M86} for this construction). Therefore
both $\alpha_1$ and $\alpha_2$ bound disks in the handlebody which we
denote by $Q_1$ and $Q_2$. We say that $Q_1$ and $Q_2$ are obtained
from $D$ by \emph{simple surgery along the returning arc $\rho$}.

The following observation is well-known (compare
\cite[Lemma~3.2]{M86}, or \cite{St02}). 
\begin{lemma}\label{lemma:unique-reduced-exchange}
If $\Sigma$ is a reduced disk system and $\rho$ is a returning arc
with endpoints on $D \in \Sigma$, then for exactly one choice of the disks
$Q_1,Q_2$ defined as above, say the disk $Q_1$, the disk system
obtained from $\Sigma$ by replacing $D$ by $Q_1$ is reduced.
\end{lemma}
\begin{proof} A reduced disk system 
equipped with an orientation 
defines a basis over $\mathbb{Z}$ 
for the relative homology group 
$H_2(V,\partial V;\mathbb{Z})=\mathbb{Z}^n$.
The homology class of the oriented disk $D$ is 
the sum of the homology
classes of the suitably oriented disks $Q_1$ and $Q_2$. 
Since $D$ is a generator of $H_2(V,\partial V;\mathbb{Z})$, 
there is exactly one of the disks $Q_1,Q_2$, 
say the disk $Q_1$, so that
the disk system $\D^\prime$ obtained from 
$\D$ by replacing $D$ by
$Q_1$ defines a basis for $H_2(V,\partial V;\mathbb{Z})$. 
Then this disk system is reduced.
\end{proof}

Note that the disk $Q_1$
is characterized by the requirement that
the two spots in the boundary of $V\setminus\Sigma$ 
corresponding to the two copies of 
$D$ are contained in distinct connected components of
$V\setminus(\Sigma\cup Q_1)$. It only depends on $\Sigma$ and the
returning arc $\rho$.

\begin{defi}\label{diskmove}
Let $\Sigma$ be a reduced disk system. 
A \emph{disk exchange move} is the replacement
of a disk $D\in \Sigma$ by a disk $D^\prime$
which is disjoint from $\Sigma$ and 
such that $(\Sigma\setminus D)\cup D^\prime$ is a reduced
disk system. If $D^\prime$ is 
determined as in Lemma~\ref{lemma:unique-reduced-exchange} by a
returning arc of a disk in a disk system $\D$ then
the modification
is called a \emph{disk exchange move of $\Sigma$
in direction of $\D$} or simply a \emph{directed
disk exchange move}.

A sequence $(\Sigma_i)$ of reduced disk systems is called a \emph{disk
exchange sequence in direction of $\D$} (or \emph{directed disk
exchange sequence}) if each $\Sigma_{i+1}$ is obtained from $\Sigma_i$
by a disk exchange move in direction of $\D$.
\end{defi}

\begin{lemma}\label{lemma:surgery-sequences}
  Let $\Sigma_1$ be a reduced disk system and let $\D$ be any other
  disk system.  
  Then there is a disk exchange sequence $\Sigma_1,\ldots,\Sigma_n$ in
  direction of $\D$ such that $\Sigma_n$ is disjoint from $\D$.
\end{lemma}
\begin{proof}
  We define the sequence $\Sigma_i$ inductively.
  Suppose $\Sigma_i$ is already defined and not yet disjoint from $\D$. Then
  there is a outermost arc $\rho$ of $\D$ with respect to
  $\Sigma_i$. By Lemma~\ref{lemma:unique-reduced-exchange}, there is a disk
  system $\Sigma_{i+1}$ obtained by a disk exchange move along this
  returning arc.  
  As a result of this surgery, the geometric intersection number between
  $\Sigma_{i+1}$ and $\D$ is strictly smaller than the geometric
  intersection number between $\Sigma_i$ and $\D$.
  Now the lemma
  follows by induction 
  on the geometric intersection number between $\partial \Sigma_1$ and
  $\partial \D$.
\end{proof}

\section{Racks}
\label{sec:racks}
In this section we define and investigate 
combinatorial objects which serve as analogs of 
train tracks for handlebodies.
Let again $V$ be a handlebody of genus $g\geq 2$,
perhaps with marked points on the boundary.

\begin{defi}\label{rack}
A \emph{rack} $R$ in $V$ is given by a reduced disk 
system $\Sigma(R)$, called
the \emph{support system} of the rack $R$,
and a collection of pairwise disjoint essential
embedded arcs in $\partial V\setminus\partial\Sigma(R)$ 
with endpoints on $\partial\Sigma(R)$,
called \emph{ropes}, which are pairwise non-homotopic 
relative to $\partial \Sigma(R)$.
At each side
of a support disk $D\in \Sigma(R)$, 
there is at least one rope which 
ends at the disk and approaches the disk from this side.

A rack $R$ is called \emph{large}, if the union of $\partial\Sigma(R)$
and the set of ropes decompose $\partial V$ into disks.
\end{defi}

Note that the
number of ropes of a rack  is uniformly bounded. 
In the sequel we often consider
isotopy classes of racks. 

Explicitly, we say 
that two racks $R,R^\prime$ are \emph{(weakly) isotopic} if their support
systems $\Sigma(R),\Sigma(R^\prime)$ are isotopic and
if after an identification of $\Sigma(R)$ with $\Sigma(R^\prime)$,
each rope of $R$ is freely homotopic relative to 
$\partial \Sigma(R)$ to a rope of $R^\prime$.
In Section~\ref{sec:rigid-racks} we will introduce a more restrictive
notion of equivalence of racks.

The handlebody group ${\rm Map}(V)$ acts
transitively on
the set of reduced disk systems,
and it acts on the set of weak isotopy classes of racks.
For every reduced disk system $\Sigma$ the stabilizer
of $\partial\Sigma$ in ${\rm Mod}(\partial V)$ is contained in 
${\rm Map}(V)$ (compare Proposition~\ref{prop:ball-group}).
This implies that there are only finitely
many orbits for the action
of ${\rm Map}(V)$ on the set of weak isotopy classes of racks. 
The stabilizer in ${\rm Map}(V)$ 
of a weak isotopy class of a rack $R$ with support system $\Sigma(R)$
contains the group $\mathbb{Z}^n$ of 
Dehn twists about the components of $\partial \Sigma(R)$.
In particular, this stabilizer is infinite.

\begin{defi}\label{carrying}
\begin{enumerate}
\item
A disk system $\D$ (or an arbitrary geodesic lamination
$\lambda$ on $\partial V$) 
is \emph{carried} by a rack $R$ if it is in minimal
position with respect to the support system 
$\Sigma(R)$ of $R$ and if each component of 
$\partial \D\setminus\partial\Sigma(R)$ (or of $\lambda\setminus\partial\Sigma(R)$)
is homotopic relative to $\partial \Sigma(R)$ 
to a rope of $R$.
\item An embedded essential arc $\rho$ in $\partial V$ 
with endpoints in $\partial \Sigma(R)$ 
is \emph{carried} by $R$ if each component of 
$\rho\setminus\partial\Sigma(R)$ is homotopic relative to
$\partial\Sigma(R)$ to a rope of $R$. 
\item A \emph{returning rope} of a rack $R$ is a rope which
begins and ends at the same side of some fixed 
support disk $D$ (i.e. defines a returning arc relative to $\partial \Sigma(R)$). 
\end{enumerate}
\end{defi}
\begin{remark}
  \begin{enumerate}[i)]
  \item A disk system $\D$ is carried by a rack $R$ if and only if
    each individual disk $D\in\D$ is carried by $R$.
  \item Every disk which does not intersect the support system
    $\Sigma(R)$ of a rack $R$ is not carried by $R$. In particular,
    the support system itself is not carried by $R$.
  \end{enumerate}
\end{remark}

Let $R$ be a rack with support system $\Sigma(R)$ and let $\alpha$ be a 
returning rope of $R$ with endpoints on a support disk $D\in \Sigma(R)$.
By Lemma~\ref{lemma:unique-reduced-exchange}, for one
of the components $\gamma_1,\gamma_2$ of 
$\partial D\setminus\alpha$, say the component $\gamma_1$, 
the simple closed curve
$\alpha\cup \gamma_1$ is the 
boundary of an embedded disk $D^\prime\subset H$
with the property that 
the disk system
$(\Sigma\setminus D)\cup D^\prime$ is reduced.

A \emph{split} of the rack $R$
at the returning rope $\alpha$ is any rack $R^\prime$
with support system $\Sigma^\prime=(\Sigma(R)\setminus D)\cup D^\prime$
whose ropes 
are given as follows.
\begin{enumerate}
\item Up to isotopy,  each 
rope $\rho^\prime$ of $R^\prime$ has its endpoints in 
$(\partial\Sigma(R)\setminus\partial D)\cup \gamma_1\subset \partial\Sigma(R)$ and 
is an arc carried by $R$. 
\item For every rope $\rho$ of $R$ there is a rope $\rho^\prime$
of $R^\prime$ such that $\rho$ is a component of 
$\rho^\prime\setminus\partial\Sigma(R)$.
\end{enumerate}

The above definition implies in particular that 
a rope of $R$ which does not have an endpoint on $\partial D$ is 
also a rope of $R^\prime$. Moreover, there is a map
$\Phi:R^\prime\to R$ which maps a rope of 
$R^\prime$ to an arc carried by $R$, and which maps 
the boundary of a support disk of $R^\prime$ to a simple closed
curve $\gamma$ of the form $\gamma_1\circ \gamma_2$ where
$\gamma_1$ either is a rope of $R$ or trivial, 
and where $\gamma_2$ is
a subarc of the boundary of a support disk of $R$
(which may be the entire boundary circle).
The image of $\Phi$ contains every rope of $R$.

Splits of racks behave differently from splits of train
tracks. Although this distinction is not explicitly needed for the
rest of this work, we note some important differences in the remainder
of this section. For these considerations we always consider racks up
to weak isotopy. 

A split of a rack $R$ at a returning rope is not unique. 
If $R^\prime$ is a split of $R$ and if $\varphi$ is a
Dehn-twist about the boundary of 
a support disk of $R$ then $\varphi(R^\prime)$ is a 
split of $R$ as well. Moreover the following example shows that even
up to the 
action of the group of Dehn twists about the boundaries of the
support system of $R$, there may be infinitely 
many racks which can be obtained from $R$ by a split.

\bigskip

\noindent
{\bf Example:} Let $V$ be the handlebody of genus $2$ and 
let $\Sigma$ be a reduced disk system consisting of two 
disks. Let $R$ be a rack with support system $\Sigma$ which 
contains 
two distinct returning ropes $\alpha,\beta$
approaching the same support disk $D\in \Sigma$ from two distinct
sides. Let $E\subset V$ be an essential
disk carried by $R$ with the following property. 
There is an
outermost component $E^\prime$ of 
$E\setminus\Sigma$ which contains an arc homotopic to $\alpha$ in its boundary.
Attached to $E^\prime
\subset E$ is a rectangle component 
$R_\beta\subset E$ of $E\setminus\Sigma$ with two opposite sides 
on $D$ which is a thickening of the returning rope $\beta$.
The rectangle $R_\beta$ is attached to a rectangle $R_\alpha$ 
with two sides on $D$ which is a thickening of 
$\alpha$ looping about the half-disk $E^\prime$.
$R_\alpha$ in turn is 
attached to a second copy of $R_\beta$ etc (see the figure).
\begin{figure}[ht]
\begin{center}
\includegraphics[width=0.8\textwidth]{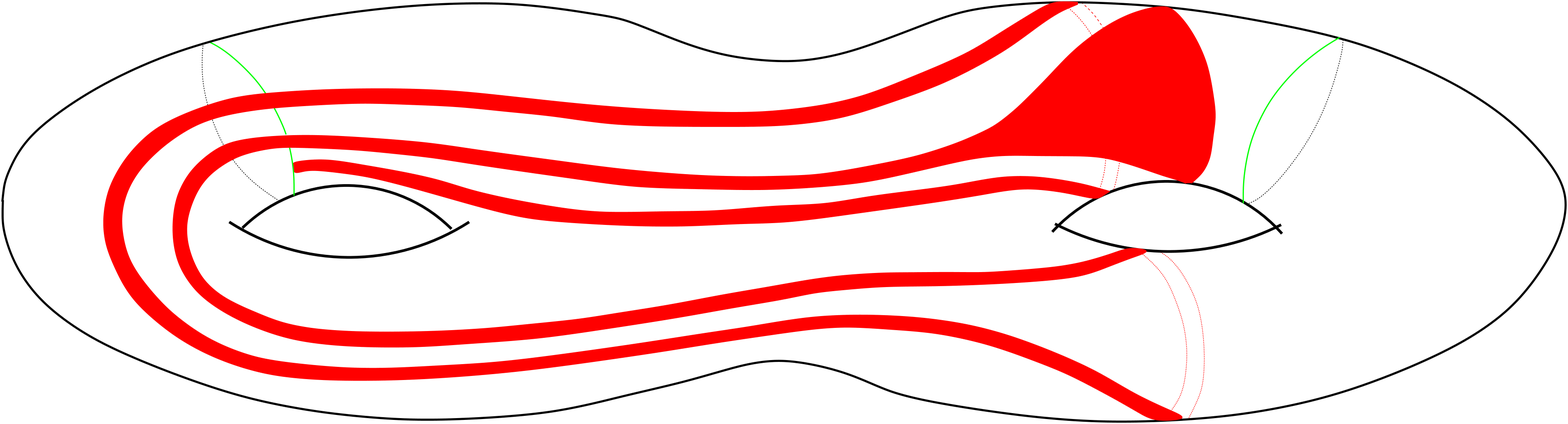}
\end{center}
\end{figure}
A rack $R^\prime$ 
whose support system is obtained from $\Sigma$ by
a single disk exchange in direction of $E$ and which carries
$\partial E$ contains a returning rope $\rho$ which
is carried
by $R$ and so that $\rho\setminus\Sigma$ has an arbitrarily large
number of components.

\bigskip

Another important difference between racks and 
train tracks concerns the relation between carrying and splitting.
One the one hand, there are splits $R'$ of $R$ which carry
disks which are not carried by $R$. 
Namely, let $R$ be a rack and $R'$ be a split of $R$. Denote the
support disk of $R'$ which is not a support disk of $R$ by $D$.
In particular, if $D$ is a disk carried by both $R$ and $R'$, then
images of $D$ under arbitrary powers of the Dehn twist about $\partial
D$ are still carried by $R'$, but not necessarily by $R$.

On the other hand, let $D$ be a disk carried by a rack $R$.
Then there may be no split $R'$ of $R$ which still carries
$D$. Namely, $R$ may have a single returning rope $\rho$ and thus
every split of $R$ has the same support system $\Sigma'$. If $\Sigma'$
is disjoint from $D$, no rack with support system $\Sigma'$ carries
$D$.

\section{The graph of rigid racks}
\label{sec:rigid-racks}
In this section we construct a geometric model for the handlebody group.
By a geometric model
we mean a connected locally finite graph on which the handlebody
group acts properly and cocompactly as a group of automorphisms.
The construction is similar in spirit to the construction of the
train track graph in \cite{H09}, which is a geometric model for the 
mapping class group.
The model we construct admits a family of distinguished paths which are 
used for a coarse geometric control of the handlebody group. These
paths are constructed below in Lemmas~\ref{lemma:force-to-carry} and~\ref{lemma:inductive-distance-control}.

As a first step one can define a \emph{graph of racks $\mathcal{R}(V)$}
in direct analogy to the definition of the train track graph in \cite{H09}. 
The vertex set of $\mathcal{R}(V)$
is the set of weak isotopy classes of large racks (satisfying a suitable
completeness condition which is not important for the current work). 
Two such vertices
are connected by an edge of length one if the corresponding 
racks are related by a single split.
By construction, the handlebody group acts on $\mathcal{R}(V)$ as
a group of automorphisms.
Imitating the proof of connectivity for the train track graph from
\cite[Corollary~3.7]{H09} one can then show 
that $\mathcal{R}(V)$ is connected. Since this result is not needed
in the sequel we do not include a proof here.

The graph of racks defined in this way is not a geometric
model for the handlebody group,  as the stabilizer of a weak isotopy
class of a rack
contains the group generated by Dehn twists about the support
system, and thus is in particular infinite. For the same reason, the
graph of racks is locally infinite. Also recall that even up to the
action of the group of Dehn twists about the support system of $R$,
there may be infinitely many different racks
which can be obtained from $R$ by a single split (as demonstrated by
the example in Section \ref{sec:racks}).

To define a geometric model for the handlebody group using racks, we
therefore have to overcome two 
difficulties. On the one hand, we need to record twist parameters
at the support curves so that the
stabilizer of a rack with a set of such twist parameters becomes finite. 
On the other hand, the edges have to be more restrictive than splits so 
that the graph becomes locally finite. 

For the purposes of this article, these problems will be addressed by
considering a more restrictive notion of equivalence of racks.
\begin{defi}
  \begin{enumerate}[i)]
  \item Let $R$ be a large rack. The union of the support system 
    and the system of ropes of $R$ defines a cell decomposition of
    the surface $\partial V$ which we call the \emph{cell
      decomposition induced by $R$}.
  \item Let $R$ and $R'$ be racks. We say that $R$ and $R'$ are
    \emph{rigidly isotopic} if the cell decompositions induced by $R$
    and $R'$ are isotopic as cell decompositions of the surface
    $\partial V$.
  \end{enumerate}
\end{defi}
In particular, if $\varphi$ is a simple Dehn twist about the boundary
of a support curve of a rack $R$, then $R$ and $\varphi^n(R)$ are not
rigidly isotopic for $n\geq 2$. This observation and the fact that the
stabilizer of a reduced disk system in the mapping class group is
contained in the handlebody group imply the following.

\begin{cor}\label{cor:marked-finite-quotient}
  The handlebody group acts on the set of rigid isotopy classes of
  racks with finite quotient and finite stabilizers.
\end{cor}
This corollary shows that the set of rigid isotopy classes of racks
can be used as the set of 
vertices of a $\Map(V)$-graph which is a geometric model for
$\Map(V)$.

To define a suitable set of edges for such a graph we note the
following lemma. 
\begin{lemma}\label{lemma:connectivity-constant}
  \begin{enumerate}[i)]
  \item There is a number $K_1>0$ with the following property. Let $R, R'$
  be two racks sharing the same support system. Then there is a
  sequence
  $$R = R_1,\ldots, R_N = R'$$
  of racks, such that the number of intersections between the cell
  decompositions induced by $R_i$ and $R_{i+1}$ is less than $K_1$ for all
  $i=1,\ldots, N-1$.
  
  \item There is a number $K_2>0$ with the following property.
    Let $R$ be a rack and let $\alpha$ be a returning rope of
    $R$. Then there is a rack $R'$ which is obtained from $R$ by a
    split along $\alpha$ such that the number of intersections between
    the cell decompositions induced by $R$ and $R'$ is less than
    $K_2$.
  \end{enumerate}
\end{lemma}
\begin{proof}
  Part \textit{i)} of the lemma follows immediately from the fact that
  for every reduced disk system $\Sigma$ of $V$, the stabilizer of 
  $\partial \Sigma$ in the mapping class group of $\partial V$ is contained
  in the handlebody group and acts with finite quotient on the set of
  all rigid isotopy classes of racks with a common support system.

  To prove part \textit{ii)}, let $\Sigma'$ be the reduced disk
  system obtained from the support system of $R$ by the disk
  exchange along the returning rope $\alpha$. Every component of
  $\partial \Sigma'$ is homotopic to a union of uniformly few edges
  of the cell decomposition induced by $R$.
  Therefore, the number of
  intersections between $\Sigma'$ and the cell decomposition induced by
  $R$ can be uniformly bounded.
  Now the claim follows as in part \textit{i)} since the stabilizer
  of $\partial \Sigma'$ in the mapping class group of $\partial V$
  is contained in the handlebody group.
\end{proof}
\begin{defi}
  The \emph{graph of rigid racks} $\mathcal{RR}(V)$ is the graph whose
  vertex set is the set of rigid isotopy classes of large racks. Two such
  vertices are joined by an edge if the intersection number between
  the cell decompositions induced by the large racks corresponding to the
  edges is at most $K$. Here $K$ is the maximum of the constants $K_1$
  and $K_2$ of Lemma~\ref{lemma:connectivity-constant}.
\end{defi}
\begin{remark}
  Part \textit{ii)} of Lemma~\ref{lemma:connectivity-constant} can be
  interpreted as the fact that twisting data about the
  support system of a rack $R$ determines a finite number of splits which are
  adapted to these twist parameters. Furthermore, each of these possible splits
  carries a coarsely unique set of twist parameters induced by the
  original rack $R$.
\end{remark}
Lemma~\ref{lemma:connectivity-constant} implies that
$\mathcal{RR}(V)$ is connected. Since the handlebody group acts on
$\mathcal{RR}(V)$ properly discontinuously and cocompactly, it is a
geometric model of the handlebody group by the Svar\'c-Milnor-Lemma.

As a next step we define a distinguished class of paths in 
$\mathcal{RR}(V)$. These paths are sufficiently well-behaved to obtain a
coarse geometric control for the handlebody group.
The length estimates for these paths use markings and
Corollary~\ref{cor:bounding-cell-intersections} which relates
word norms of mapping class group elements to intersection numbers of
cell decompositions. The necessary
definitions and statements are given in the Appendix.

In order to simplify the notation for the rest of the paper, we usually do not
specify constants or additive and multiplicative errors in formulas, but
rather state that a quantity $x$ is ``coarsely
bounded'' by some other quantity $y$ (or ``uniformly bounded''). By
this we mean that there are constants $C_1,C_2$
which only depend on the genus (and the number of marked points) of
$V$, such that $x$ is bounded by $C_1\cdot y+C_2$ (or $C_1$).
\begin{lemma}\label{lemma:force-to-carry}
  There is a number $k>0$ satisfying the following.
  Let $P$ be a pants decomposition of $\partial V$ all of whose
  components bound disks in $V$. Let $R$ be a large rack with support
  system $\Sigma(R)$. Then there is a large rack $R'$ with the
  following properties.
  \begin{enumerate}[i)]
  \item The support system $\Sigma(R')$ of $R'$ agrees with the one of $R$.
  \item Each component of $P$ which intersects the support system of
    $R$ essentially is carried by $R'$.
  \item Each component of $P\setminus\partial \Sigma(R')$ intersects the
    cell decomposition induced by $R'$ in at most 
    $k$ points.
  \item The distance between $R$ and $R'$ in $\mathcal{RR}(V)$ is
    coarsely bounded by $i(P, \partial\Sigma(R))$.
  \end{enumerate}
\end{lemma}
\begin{proof}
  Denote the cell decomposition induced by $R$ by $C$.
  Let $S'$ be the surface obtained from $\partial V$ by cutting at
  $\partial\Sigma(R)$. The intersection of $P$ with $S'$ is a union of
  simple closed curves and arcs connecting the boundary components of
  $S'$. We call these arcs the \emph{arcs induced by $P$}.
  Let $\hat{R}$ be the rack whose support support system agrees with  
  the one of $R$ and whose ropes are given by the arcs induced by $P$.
  If $\hat{R}$ is not a large rack, then we can add ropes to $\hat{R}$
  which intersect $P$ in uniformly few points, and which intersect
  ropes of $R$ in at most $i(P,C)$ points. Call the result $R'$.

  From the construction of the rack $R'$, properties \textit{i)} to
  \textit{iii)} are immediate. Property \textit{iv)} follows by
  applying Corollary \ref{cor:bounding-cell-intersections} to the cell
  decomposition $C$ and the cell decomposition induced by $R'$ on the
  subsurface $S'$.
\end{proof}
\begin{defi}
  If $P$ and $R'$ satisfy the conclusions \textit{ii)} and \textit{iii)}
  of Lemma~\ref{lemma:force-to-carry} above, we say that \emph{$P$ is
    effectively carried by $R'$}.
\end{defi}
The following lemma is the main step towards the upper distortion
bound for the handlebody group and contains the construction of the
distinguished paths in the handlebody group.
\begin{lemma}\label{lemma:inductive-distance-control}
  Let $P$ be a pants decomposition all of whose components bound disks
  in $V$.
  Suppose $P$ is effectively carried by a rack $R$ with support system
  $\Sigma(R)$. If at least one component of $P$ intersects
  $\partial\Sigma(R)$ essentially, there is a rack $R'$ with the
  following properties.
  \begin{enumerate}[i)]
  \item The support system $\Sigma(R')$ is obtained from $\Sigma(R)$
    by a disk exchange move in the direction
    of a component of $P$.
  \item $P$ is effectively carried by $R'$.
  \item The distance of $R$ and $R'$ in $\mathcal{RR}(V)$ is coarsely
    bounded by $i(P,\partial\Sigma(R))$.
  \end{enumerate}
\end{lemma}
\begin{proof}
  Since the intersection of $P$ with $\partial\Sigma(R)$ is nonempty,
  the rack $R$ has a returning rope $\alpha$ corresponding to an arc
  induced by $P$. 

  Let $\Sigma'$ be the reduced disk system obtained from $\Sigma(R)$
  by a disk exchange along the returning leaf $\alpha$. Each component
  of $\partial \Sigma'$ intersects the cell decomposition induced by
  $R$ in uniformly few points. Define a rack $\hat{R}$ with support
  system $\Sigma'$ by choosing the arcs induced by $P$ relative to
  $\Sigma'$ as ropes. By construction, each rope of $\hat{R}$ is
  obtained as a concatenation of ropes of $R$ (as in the definition of
  the split of a rack). Furthermore, each rope of $\hat{R}$ intersects
  $\Sigma(R)$ in at most as many points as $P$ does. Therefore, the
  intersection number between a rope of $\hat{R}$ and the cell
  decomposition induced by $R$ can be coarsely bounded by $i(P,\partial\Sigma(R))$.
  We can extend $\hat{R}$ in any way to a large rack $R'$ such that every
  rope of $R'$ has the same property.
  Both $R'$ and $R$ intersect $\partial \Sigma'$ in uniformly few
  points. The mapping class group of $\partial V\setminus \partial
  \Sigma'$ is contained in the handlebody group and undistorted in the
  mapping class group. Hence Corollary~\ref{cor:bounding-cell-intersections}
  applied in the subsurface $\partial V\setminus\partial\Sigma'$
  implies that the distance
  between $R$ and $R'$ in $\mathcal{RR}(V)$ is coarsely bounded by
  $i(P,\partial\Sigma(R))$. 
  Now we can apply Lemma~\ref{lemma:force-to-carry} to $R'$ to obtain
  a rack with the desired properties.
\end{proof}

The following theorem is an easy consequence of
Lemma~\ref{lemma:inductive-distance-control}.
\begin{theorem}\label{thm:upper-distortion-bound}
  Let $g\geq 2$ be arbitrary. Then the handlebody group $\Map(V_g)$ is
  at most exponentially distorted in the mapping class group.
\end{theorem}
Together with the results from Sections~\ref{sec:marked-points}
and~\ref{sec:closed-surfaces} this theorem implies the main theorem from the
introduction.
\begin{proof}[Proof of Theorem~\ref{thm:upper-distortion-bound}]
  There is a number $K>0$ such that for every large rack $R$ there is
  a pants decomposition $P_R$ whose geometric intersection number with
  the cell decomposition $C(R)$ induced by $R$ is
  bounded by $K$. This is due to the fact that the handlebody group acts
  cocompactly on the graph of rigid racks.

  Let $R_0$ be a rack, and $P_0$ such a pants decomposition. Let $f$
  be an arbitrary element of the handlebody group. Put $P =
  f(P_0)$. By
  Proposition~\ref{prop:intersection-number-versus-distance} the
  geometric intersection number between $P$ and $P_0$ is coarsely
  bounded exponentially in the word norm of $f$ in the mapping class
  group. Denote this bound by $N$.

  As a first step, apply Lemma~\ref{lemma:force-to-carry} to $R_0$ and
  $P$ to construct a rack $R_1$ which effectively carries $P$ and
  whose distance to $R_0$ is coarsely bounded by $N$.
  Next, use Lemma~\ref{lemma:inductive-distance-control} to construct
  a rack $R_2$ whose distance to $R_1$ is again coarsely bounded by
  $N$, and such that the number of intersections between $P$ and
  $\Sigma(R_2)$ is strictly less than the number of intersections
  between $P$ and $\Sigma(R_1)$.
  Inductively repeating this procedure we find a sequence
  $R_1,\ldots,R_K$ of racks of length $K$ coarsely bounded by $N^2$,
  and such that $P$ is disjoint from $\Sigma(R_K)$.
  In particular, there is a handlebody group element $g$ which
  maps $P_0$ to $P$ and whose word
  norm in the handlebody group is also coarsely bounded by $N^2$. The difference $f^{-1}\circ g$ fixes the pants
  decomposition $P_0$ and hence is a Dehn multitwist about $P_0$. As
  the group of Dehn multitwists about $P_0$ is contained in the
  handlebody group, and undistorted in the mapping class group, the
  word norm of $f^{-1}\circ g$ in the handlebody group is also
  coarsely bounded by $N^2$. This shows the theorem.
\end{proof}

\appendix
\section{Markings and intersection numbers}
\label{app:markings}
In this Appendix we recall some facts about markings and intersection
numbers which are used several times in this work.

Our terminology deviates slightly from the one used in \cite{MM00}, so
we also recall the necessary definitions.
\begin{defi}
  A \emph{marking} $\mu$ of a surface $S$ is a pants
  decomposition $P$ of $S$ together with a clean transversal for each
  curve in $P$.
  Here, a \emph{clean transversal} to a pants curve $\gamma\in P$ is a
  curve $c$ which is disjoint from all curves $\gamma' \in P\setminus\gamma$
  and which intersects $\gamma$ in the minimal number of points.
\end{defi}
Two clean transversals to a curve $\alpha$ in a pants decomposition
$P$ differ by a Dehn twist about $\alpha$ (after possibly applying a
half-twist about $\alpha$). In this way, the set of clean
transversals can be thought of as a twist normalization about the
pants decomposition curves.

Note that the object we denote by ``marking'' is called ``complete
clean marking'' in the terminology of \cite{MM00}. The more general
notion of marking used in \cite{MM00} does not play any role in the
present work.

Let $S$ be a oriented surface of finite type and
negative Euler characteristic (possibly with punctures and boundary
components).  
\emph{Subsurface projections to annuli} in $S$ are defined in the
following way (compare \cite{MM00}). 
Recall that the arc complex of a closed annulus $A$ is the graph whose
vertex set is the set of arcs connecting the two boundary components
of $A$ up to isotopy fixing $\partial A$ pointwise. Two such vertices
are connected by an edge of length one, if the corresponding arcs can
be realized with disjoint interior.

Let $\alpha$ be an essential simple closed curve
on $S$. By $S_\alpha$ we denote the annular cover corresponding to
$\alpha$. Explicitly, $S_\alpha$ is a covering surface of $S$ corresponding
to the (conjugacy class of the) cyclic subgroup of $\pi_1(S)$
generated by $\alpha$. Since $S$ has negative Euler characteristic, it
carries a hyperbolic metric which lifts to a hyperbolic metric on
the annulus $S_\alpha$. In particular, $S_\alpha$ has a natural
boundary compactifying it to a closed annulus.  

Let $\beta$ be a simple closed curve or essential arc on $S$
intersecting $\alpha$. Consider the set of lifts $\widetilde{\beta}$
of $\beta$ to 
$S_\alpha$ which connect the two boundary components of
$S_\alpha$. Every element of this set defines a vertex in
the arc complex of the annulus $S_\alpha$.
We call the set of all these vertices the \emph{subsurface
  projection of $\beta$ to $\alpha$}. The subsurface projection of
$\beta$ to $\alpha$ has diameter at most one as all lifts of $\beta$
to $S_\alpha$ are disjoint.

\begin{defi}
  The \emph{marking graph} of $S$ is the graph whose vertex set is the set of
  isotopy classes of markings. Two such markings $\mu$
  and $\mu'$ are joined by an edge of length one if they differ by an
  \emph{elementary move}. An elementary move from $\mu$ to $\mu'$ is
  one of the following two operations.
  \begin{enumerate}[i)]
  \item $\mu'$ has the same underlying pants decomposition as
    $\mu$. The transversals of $\mu'$ are obtained from the ones of
    $\mu$ by applying one primitive Dehn twist about one of the pants
    curves.
  \item Replace a pants curve $\alpha$ by its corresponding
    clean transversal $\beta$ in $\mu$. Then modify $\alpha$ to a
    clean transversal of $\beta$ (``cleaning the marking'' in the
    terminology of \cite{MM00}).
  \end{enumerate}
\end{defi}
The cleaning operation is described in detail in
\cite[Lemma~2.4]{MM00} (also compare the discussion on page 21 of
\cite{MM00}). 

Since the details are not relevant for the
current work, we do not review them here.
The marking graph is a connected, locally finite graph on which the
mapping class group of $S$ acts with finite point stabilizers and
finite quotient (compare \cite{MM00}). Therefore, it is quasi-isometric
to the mapping class group.

The following proposition is well-known to experts and relates
distances in the marking graph to intersection numbers. 
Since we did not find a proof in the literature,
we include one here for completeness.
\begin{prop}\label{prop:intersection-number-versus-distance}
  Let $\mu_1,\mu_2$ be markings of a surface $S$. 
  If $\mu_1$ and $\mu_2$ are of distance $k$ in the
  marking graph, then the total number of intersections between
  $\mu_1$ and $\mu_2$ is bounded exponentially in $k$.
  Conversely, the total intersection number between $\mu_1$ and
  $\mu_2$ is a coarse upper bound for the distance between $\mu_1$ and
  $\mu_2$ in the marking graph of $S$.
\end{prop}
\begin{proof}
  We begin with the lower bound for the distance in the marking graph. Let
  $\mu_1$ and $\mu_2$ be two markings. For a number $\epsilon>0$, we
  say a marked Riemann surface $X$ belongs to the $\epsilon$-thick part of
  Teichm\"uller space if the length of each simple closed geodesic on
  $X$ is at least $\epsilon$. We will simply speak of the thick part,
  if the corresponding $\epsilon$ is understood from the context.
  There are points $X_i$ the $\epsilon$-thick part
  of Teichm\"uller space for $S$ such that each curve in $\mu_i$ is shorter than
  some universal constant $C$ on $X_i$. Here, $\epsilon$ is a
  universal constant depending only on the genus of the surface $S$.
  Explicitly, let $P_i$ be the
  underlying pants decomposition of the marking $\mu_i$. The pants
  decomposition $P_i$ defines Fenchel-Nielsen coordinates for the
  Teichm\"uller space of $S$. This implies that there is a marked
  Riemann surface $X_i'$ such that each curve in $P_i$ has hyperbolic
  length $1$ on $X'_i$. On a hyperbolic pair of pants all of whose
  boundary components have 
  lengths equal one the distance between any two boundary components
  is uniformly bounded. This implies that on $X'_i$ there are clean
  transversals to $P_i$ whose hyperbolic length is also uniformly
  bounded. By changing the marking on $X_i'$ by Dehn twists about
  $P_i$ we obtain the desired surfaces $X_i$.

  If the distance between
  $\mu_1$ and $\mu_2$ in the marking graph is bounded by $k$, then the
  Teichm\"uller distance between $X_1$ and $X_2$ is also coarsely bounded by $k$
  since the mapping class group acts properly and cocompactly on the
  thick part of Teichm\"uller space.
  Thus the total hyperbolic length of $\mu_2$ on $X_1$ is bounded by $e^{2k}\cdot C$
  by Wolpert's lemma (\cite[Lemma 3.1]{W79}). But each curve
  in $\mu_1$ has a collar of definite width on $X_1$ since its length is bounded
  by $C$, and therefore
  the total number of intersections of $\mu_1$ and $\mu_2$ is also
  coarsely bounded by $e^{2k}$.

  Next we show the upper bound for the distance in the marking graph.
  In the proof we will use singular
  Euclidean structures as in \cite{B06} and the relation between the
  mapping class group of a surface and the corresponding Teichm\"uller space.

  Let $P_1$ and $P_2$ be the underlying pants decompositions of the markings
  $\mu_1,\mu_2$. We may assume that $P_1 \union P_2$ fills the
  surface, i.e. that all components of $S\setminus (P_1\union P_2)$
  are simply connected. Namely, if $P_1 \union P_2$ does not
  fill, then $P_1$ and $P_2$ share a common curve $\alpha$. 
  We can then change the transversal to $\alpha$ in $\mu_1$ such that
  the diameter of the subsurface projection to $\alpha$ of the
  transversals to $\alpha$ in $\mu_1$ and $\mu_2$ is at most one.
  The number of steps necessary for
  this modification is bounded by the intersection number between the two
  transversals.
  We can then pass to the subsurface obtained by cutting $S$ along the common
  curve $\alpha$ and discarding the corresponding transversal. Repeat
  this procedure until $P_1\union P_2$ fills. 

  Furthermore, we can assume that the twist about a pants curve
  $\delta\in P_1$ defined by $\mu_1$ coarsely agrees with the one defined by
  $P_2$. By this we mean the following. Since $P_1$ and $P_2$ fill the
  surface, there is at least one curve of $P_2$ which
  intersects $\delta$. Denote by $c_\delta$ the transversal to
  $\delta$ in $\mu_1$. 
  The diameter of the subsurface projection of $P_2$ and $c_\delta$ to
  $\delta$ is bounded from above by the intersection number between 
  $\mu_1$ and $\mu_2$. Hence, after modifying the transversal to
  $\delta$ in $\mu_1$ by at most $i(\mu_1,\mu_2)$ 
  Dehn twists about $\delta$, the diameter of the projection is at 
  most $3$. 
  Similarly, we modify $\mu_2$ such that the twist about the pants 
  curves in $P_2$ given by $\mu_2$ agrees with the one defined by $P_1$.

  For a pair of measured laminations $\lambda_1,\lambda_2$ which
  jointly fill the surface and satisfy
  $i(\lambda_1,\lambda_2)=1$ we denote by $q(\lambda_1,\lambda_2)$ the
  quadratic differential whose horizontal measured lamination is $\lambda_1$
  and whose vertical measured lamination is $\lambda_2$.
  Now let $\rho$ be the Teichm\"uller geodesic defined by $P_1$ and
  $P_2$; that is $\rho_t=q(e^{-t} P_1, e^{t}/i(P_1,P_2) P_2)$
  (compare the construction in \cite{B06} for pairs of curves).
  Recall that on every hyperbolic surface of genus $g$ there is a
  pants decomposition such that the hyperbolic length of each pants
  curve is bounded by a universal constant $B$ (the Bers constant)
  which depends only on the genus. By the collar lemma, a curve whose
  hyperbolic length is bounded by $B$ has extremal length coarsely bounded by $B$.
  Thus the length of such a curve in any singular Euclidean 
  metric in the same conformal class is bounded by a universal
  constant $B'$.

  We set $T = \log(2B')$. Then for the singular Euclidean
  metric defined by $\rho_{-T}$, a curve
  whose length is smaller than $B'$ cannot intersect $P_1$. Hence,
  $P_1$ is the only Bers short pants decomposition for
  $\rho_{-T}$. Similarly, $P_2$ is the only Bers short pants
  decomposition on $\rho_{\log(i(P_1,P_2))+T}$.
  In particular, there are two points $X_1, X_2$ in Teichm\"uller
  space, whose Teichm\"uller distance is bounded by $2T +
  \log(i(P_1,P_2))$ and such that $P_i$ is Bers short on $X_i$.

  Now for any $k$ which is sufficiently large, by work of Rafi we have
  the following estimate for the Teichm\"uller distance
  $d_\teich(X_1,X_2)$ (compare \cite[Equation~(19)]{R07}).
  $$d_\teich(X_1, X_2) \succ \sum_Y\left[d_Y(\mu'_1,\mu'_2)\right]_k +
  \sum_{\alpha \notin \Gamma}\log\left[d_\alpha(\mu'_1,\mu'_2\right]_k.$$
  Here, $\mu'_1$ and $\mu'_2$ are shortest markings on $X_1$ and $X_2$,
  respectively, and $\left[x\right]_k$ is a cutoff function which is
  $0$ if $x\leq k$ and $x$ otherwise. The expression $a \succ b$ means 
  that $a$ is coarsely bounded by $b$.
  The first sum is taken over all subsurfaces $Y\subset S$, while the
  indexing set
  $\Gamma$ of the second sum is the set of (isotopy classes of) simple
  closed curves which 
  are short on either $X_1$ or $X_2$. Note that in our case $\Gamma$
  agrees with the union of the pants curves in $P_1$ and $P_2$.
  In both cases $d_Y$ (or
  $d_\alpha$) denotes the diameter of the set of subsurface
  projections of $\mu'_1$ and $\mu'_2$ to $Y$ (or $\alpha$).

  In our case, since $P_1$ and $P_2$ fill, we can replace the
  subsurface projections of $\mu'_i$ by those of $P_i$, except maybe in
  the cases where the subsurface is bounded by curves
  contained in $\Gamma$. Hence we get
  $$d_\teich(X_1, X_2) \succ \sum_{\partial Y \not\subset \Gamma}\left[d_Y(P_1,P_2)\right]_k +
  \sum_{\alpha \notin \Gamma}\log\left[d_\alpha(P_1,P_2)\right]_k.$$
  Now, since $d_\teich(X_1,X_2)\prec \log(i(P_1,P_2))$ we have
  $$i(P_1,P_2) \succ \sum_{\partial Y \not\subset \Gamma}\left[d_Y(P_1,P_2)\right]_k +
  \sum_{\alpha \notin \Gamma}\left[d_\alpha(P_1,P_2)\right]_k.$$
  Since the number of subsurfaces whose boundary is completely
  contained in $\Gamma$ is uniformly bounded, and the total
  intersection of $\mu_1$ and $\mu_2$ bounds each of these projections,
  we get
  $$i(\mu_1,\mu_2) \succ \sum_{Y}\left[d_Y(\mu_1,\mu_2)\right]_k +
  \sum_{\alpha}\left[d_\alpha(\mu_1,\mu_2)\right]_k.$$
  where now the sums are taken over all subsurfaces and all curves
  respectively. By \cite[Theorem~6.12]{MM00}, the right hand side
  of this inequality is coarsely equal to the distance of $\mu_1$ and
  $\mu_2$ in the marking graph. This shows the first claim.
\end{proof}

In the proof of the upper bound on distortion of the handlebody group
the following corollary is used in an essential way.
\begin{cor}\label{cor:bounding-cell-intersections}
  Let $N>0$ be given. Let $C$ be a cell decomposition of the surface
  $S$ with at most $N$ cells. Let $f \in \Map(S)$ be arbitrary. The
  intersection number between $C$ and $f(C)$ is coarsely bounded by
  an exponential of the word norm of $f$. Here, the constants depend on the genus of $S$
  and the number $N$.

  Similarly, let $C$ and $C'$ are cell decomposition with at most $N$
  cells and which intersect in $K$ points. Then there is a mapping
  class $g$ whose word norm is bounded coarsely in $K$, and such that
  $g(C)$ and $C'$ intersect in uniformly few points.
\end{cor}
\begin{proof}
  Note that up to the action of the mapping class group there are only
  finitely many cell decompositions $C$ of $S$ with at most $N$
  cells. Hence, there is a constant $K>0$ such that for any such cell
  decomposition $C$ there is a marking $\mu_C$ whose intersection
  number with $C$ is bounded by $K$.

  By the preceding
  Proposition~\ref{prop:intersection-number-versus-distance} the
  number of intersections between $\mu_C$ and $f(\mu_C)$ is coarsely
  bounded exponentially in the word norm of $f$.
  Since the intersection number between $f(\mu_C)$ and $f(C)$ is
  uniformly bounded, the corollary follows.

  Similarly, if $C$ and $C'$ intersect in $K$ points, then the
  intersection number between $\mu_C$ and $\mu_{C'}$ can be coarsely
  bounded by $K$. Hence,
  Proposition~\ref{prop:intersection-number-versus-distance} implies
  the second claim of the corollary.
\end{proof}

\end{document}